\documentclass[11pt, a4paper]{article}
\usepackage{hyperref}
\usepackage{latexsym}
\usepackage{amsmath, epsfig}
\usepackage{color}
\usepackage{graphicx, float}
\usepackage{amssymb,amsfonts,amsthm,amsmath,graphicx,url}

\usepackage{color}

\definecolor{Thistle}{rgb}{0.847,0.749,0.847}
\definecolor{Khaki}{rgb}{0.941,0.902,0.549}
\definecolor{Orchid}{rgb}{0.855,0.439,0.839}
\definecolor{MediumOrchid}{rgb}{0.729,0.333,0.827}

\definecolor{brown}{rgb}{0.8,0.5,0}
\definecolor{LightBrown}{rgb}{0.8,0.2,0.4}

\definecolor{DarkGray}{rgb}{0.78,0.78,0.78}
\definecolor{DarkMidGray}{rgb}{0.81,0.81,0.81}
\definecolor{MidGray}{rgb}{0.85,0.85,0.85}
\definecolor{LightGray}{rgb}{0.88,0.88,0.88}
\definecolor{VeryLightGray}{rgb}{0.96,0.96,0.96}

\definecolor{GrayA}{rgb}{0.7,0.7,0.7}
\definecolor{GrayB}{rgb}{0.78,0.78,0.78}
\definecolor{GrayC}{rgb}{0.80,0.80,0.80}
\definecolor{GrayD}{rgb}{0.82,0.82,0.82}
\definecolor{GrayE}{rgb}{0.84,0.84,0.84}
\definecolor{GrayF}{rgb}{0.86,0.86,0.86}
\definecolor{GrayG}{rgb}{0.88,0.88,0.88}
\definecolor{GrayH}{rgb}{0.90,0.90,0.90}
\definecolor{GrayI}{rgb}{0.92,0.92,0.92}
\definecolor{GrayJ}{rgb}{0.94,0.94,0.94}

\definecolor{VeryLightBlue}{rgb}{0.9,0.9,1}
\definecolor{LightBlue}{rgb}{0.8,0.8,1}
\definecolor{MidBlue}{rgb}{0.5,0.5,1}
\definecolor{DarkBlue}{rgb}{0,0,0.6}

\definecolor{Gold}{rgb}{1,0.843,0}

\definecolor{LightGreen}{rgb}{0.88,1,0.88}
\definecolor{MidGreen}{rgb}{0.6,1,0.6}
\definecolor{DarkGreen}{rgb}{0,0.6,0}

\definecolor{VeryLightYellow}{rgb}{1,1,0.9}
\definecolor{LightYellow}{rgb}{1,1,0.6}
\definecolor{MidYellow}{rgb}{1,1,0.5}
\definecolor{DarkYellow}{rgb}{1,1,0.2}

\definecolor{VeryLightRed}{rgb}{1,0.9,0.9}
\definecolor{LightRed}{rgb}{1,0.8,0.8}
\definecolor{MidRed}{rgb}{1,0.55,0.55}

\long\def\delete#1{}

\newtheorem{theorem}{Theorem}[section]
\newtheorem{lemma}[theorem]{Lemma}
\newtheorem{corollary}[theorem]{Corollary}
\newtheorem{definition}{Definition}[section]

\newtheorem{conjecture}[theorem]{Conjecture}

\input amssym.def
\input amssym.tex

\newcommand{\be}{\begin{equation}}
\newcommand{\ee}{\end{equation}}
\newcommand{\bea}{\begin{eqnarray}}
\newcommand{\eea}{\end{eqnarray}}
\newcommand{\bean}{\begin{eqnarray*}}
\newcommand{\eean}{\end{eqnarray*}}

\usepackage{xcolor}
\usepackage[normalem]{ulem}

\title{Hadwiger's Conjecture for Squares of $2$-trees}

\author{L. Sunil Chandran$^{1}$\thanks{Part of the work was done when this author was visiting Max Planck Institute for Informatics, Saarbruecken, Germany supported by Alexander von Humboldt Fellowship.}, Davis Issac$^{2}$ and Sanming Zhou$^{3}$\thanks{Research supported by ARC Discovery Project DP120101081.} \\ \\
{\small $^{1}$ Indian Institute of Science, Bangalore -560012, India} \\ 
{\small \texttt{sunil@csa.iisc.ernet.in}} \smallskip \\
{\small $^{2}$ Max Planck Institute for Informatics, Saarland Informatics Campus, Germany} \\ 
{\small \texttt{dissac@mpi-inf.mpg.de}} \smallskip \\
{\small $^{3}$ School of Mathematics and Statistics, The University of Melbourne, Parkville, VIC 3010, Australia}\\ 
{\small \texttt{sanming@unimelb.edu.au}}}

\date{}

\begin{document}
\openup 0.4\jot
\maketitle 

\begin{abstract}
Hadwiger's conjecture asserts that any graph contains a clique minor with order no less than the chromatic number of the graph. We prove that this well-known conjecture is true for all graphs if and only if it is true for squares of split graphs. This observation implies that Hadwiger's conjecture for squares of chordal graphs is as difficult as the general case, since chordal graphs are a superclass of split graphs. Then we consider 2-trees which are a subclass of each of planar graphs, 2-degenerate graphs and chordal graphs. We prove that Hadwiger's conjecture is true for squares of $2$-trees. We achieve this by proving the following stronger result: for any $2$-tree $T$, its square $T^2$ has a clique minor of order $\chi(T^2)$ for which each branch set induces a path, where $\chi(T^2)$ is the chromatic number of $T^2$.


\end{abstract}

\section{Introduction} 
\label{sec:introduction}

A graph $H$ is called a \emph{minor} of a graph $G$ if a graph isomorphic to $H$ can be obtained from a subgraph of $G$ by contracting edges. An \emph{$H$-minor} is a minor isomorphic to $H$, and a \emph{clique minor} is a $K_t$-minor for some positive integer $t$, where $K_t$ is the complete graph of order $t$. The \emph{Hadwiger number} of $G$, denoted by $\eta(G)$, is the largest integer $t$ such that $G$ contains a $K_t$-minor. A graph is called \emph{$H$-minor free} if it does not contain an $H$-minor. 
The \emph{chromatic number} of $G$, denoted by $\chi(G)$, is the least positive integer $k$ such that $G$ is \emph{$k$-colorable}, in the sense that $k$ colors are sufficient to color the vertices of $G$ such that adjacent vertices receive different colors. 

In 1937, Wagner \cite{wagner} proved that the Four Color Conjecture is equivalent to the following statement: If a graph is $K_5$-minor free, then it is $4$-colorable. In 1943, Hadwiger \cite{hadwiger1943klassifikation} proposed the following conjecture which is a far reaching generalization of the Four Color Theorem.
\begin{conjecture}
For any integer $t \ge 1$, every $K_{t+1}$-minor free graph is $t$-colorable; that is, $\eta(G) \ge \chi(G)$ for any graph $G$. 	
\end{conjecture}

Hadwiger's conjecture is well known to be a challenging problem. 
Bollob\'{a}s, Catlin and Erd\H{o}s \cite{Bollobs1980195} describe it as ``one of the deepest unsolved problems in graph theory''. Hadwiger himself \cite{hadwiger1943klassifikation} proved the conjecture for $t = 3$. (The conjecture is trivially true for $t=1,2$.) In view of Wagner's result \cite{wagner}, Hadwiger's conjecture for $t = 4$ is equivalent to the Four Color Conjecture, the latter being proved by Appel and Haken \cite{appel1977everyi,appel1977everyii} in 1977. In 1993,  Robertson, Seymour and Thomas \cite{robertson1993hadwiger} proved that Hadwiger's conjecture is true for $t = 5$. The conjecture remains unsolved for $t \ge 6$, though for $t=6$ Kawarabayashi and Toft~\cite{kawarabayashi2005any} proved that any graph that is $K_7$-minor free and $K_{4,4}$-minor free is $6$-colorable.
 
Similar to other difficult conjectures in graph theory, attempting Hadwiger's conjecture for some natural graph classes may lead to new techniques and shed light on the general case. So far Hadwiger's conjecture has been proved for several classes of graphs, including line graphs \cite{reed2004hadwiger}, proper circular arc graphs \cite{belkale2009hadwiger}, quasi-line graphs \cite{Chudnovsky:2008:HCQ:1400140.1400143}, 3-arc graphs \cite{WXZ}, complements of Kneser graphs \cite{XZ}, and powers of cycles and their complements \cite{li2007hadwiger}. There is also an extensive body of work on the Hadwiger number; see, for example, \cite{chandran2008hadwiger} and \cite{GHPP}.  

As mentioned above, Reed and Seymour \cite{reed2004hadwiger} proved that Hadwiger's conjecture is true for line graphs. Recently, there have been multiple attempts to generalize this result to graph classes that properly contain all line graphs. This was typically achieved by identifying some features of line graphs and using them as defining properties of the super class. An important super class of line graphs introduced in \cite{CS}, for which Hadwiger's conjecture has been confirmed \cite{Chudnovsky:2008:HCQ:1400140.1400143}, is the class of \emph{quasi-line graphs}, which are graphs with the property that the neighborhood of every vertex can be partitioned into at most two cliques.

Our research for this paper began with an unsuccessful attempt to generalize the result above for quasi-line graphs by considering classes of graphs with the property that the neighborhood of every vertex can be partitioned into a small number of cliques. A natural choice for us was the class of square graphs of bounded degree graphs, where the \emph{square} of a graph $G$, denoted by $G^2$, is the graph with the same vertex set as $G$ such that two vertices are adjacent if and only if the distance between them in $G$ is equal to $1$ or $2$. It is readily seen that in $G^2$ the neighborhood of each vertex $v$ can be partitioned into at most $d_G(v)$ cliques, where $d_G(v)$ is the degree of $v$ in $G$. However, we soon realized that proving Hadwiger's conjecture for squares of even split graphs is as difficult as proving it for all graphs, where a graph is \emph{split} if its vertex set can be partitioned into an independent set and a clique. This observation is our first result whose proof is straightforward and will be given in the next section.
 
\begin{theorem}
\label{thm:hadsquare}
Hadwiger's conjecture is true for all graphs if and only if it is true for squares of split graphs.
\end{theorem}

A graph is called a \emph{chordal graph} if it contains no induced cycles of length at least $4$. Since split graphs form a subclass of the class of chordal graphs, Theorem \ref{thm:hadsquare} implies:

\begin{corollary}
\label{coro:hadsquare}
Hadwiger's conjecture is true for all graphs if and only if it is true for squares of chordal graphs. 
\end{corollary}

Theorem \ref{thm:hadsquare} and Corollary \ref{coro:hadsquare} suggest that squares of chordal or split graphs may capture the complexity of Hadwiger's conjecture. These are curious results for us, though they may not make Hadwiger's conjecture easier to prove. Nevertheless, the availability of the property of being square of a split or chordal graph may turn out to be useful. Moreover, Theorem \ref{thm:hadsquare} motivates the study of Hadwiger's conjecture for squares of graphs. In particular, in light of Corollary \ref{coro:hadsquare}, it would be interesting to study Hadwiger's conjecture for squares of some interesting subclasses of chordal graphs in the hope of getting new insights into the conjecture. As a step towards this, we prove that Hadwiger's conjecture is true for squares of $2$-trees. Before presenting this result, let us explain why it is interesting to consider squares of $2$-trees.  

Chordal graphs are precisely the graphs that can be constructed by recursively applying the following operation a finite number of times beginning with a clique: Choose a clique in the current graph, introduce a new vertex, and make this new vertex adjacent to all vertices in the chosen clique. If we begin with a $k$-clique and choose a $k$-clique at each step, then the graph constructed this way is called a \emph{$k$-tree}, where $k$ is a fixed  positive integer. 

We call a graph $G$ a \emph{2-simplicial graph} if $V(G)$ has an ordering such that the higher numbered neighbors of each vertex can be partitioned into at most two cliques. 


It can be easily verified that all quasi-line graphs are 2-simplicial  graphs, but the converse is not true. Thus, in view of the above-mentioned result for quasi-line graphs \cite{Chudnovsky:2008:HCQ:1400140.1400143}, it would be interesting to study whether Hadwiger's conjecture is true for all 2-simplicial graphs. Considering the effort \cite{Chudnovsky:2008:HCQ:1400140.1400143} required for quasi-line graphs, resolving Hadwiger's conjecture for
2-simplicial graphs is likely to be a difficult task.
Moreover, the class of circular arc graphs is a proper subclass of
2-simplicial graphs
\footnote {Consider an ordering of the vertices of a circular arc graph such that a vertex $u$ with a smaller arc always gets a smaller number.}
and as far as we know a lot of effort has already gone into proving Hadwiger's conjecture for circular arc graphs, without success. Therefore, before attempting the entire class of 2-simplicial graphs it seems rational to start with some different but  interesting subclasses of 2-simplicial graphs. Viewing from the context of  the squaring operation of  graphs, we asked the following question: Is there a subclass of  2-simplicial
graphs
which can be expressed as the square of some natural class of graphs? 
If $u \in V(G)$ and $u_1,u_2,\ldots,u_t \in N_G(u)\cap H(u)$ 
(where $H(u)$ is the vertices of $G$ that are
higher numbered than $u$ with respect to the $2$-simplicial ordering), it is clear that in $G^2$, $\cup_{i} (H(u) \cap N_G[u_i])$ will be a subset of the higher numbered neighbors of $u$ in $G^2$.
For each $u_i$, $N_G[u_i]$ will form a clique in $G^2$ but there is no reason why  higher numbered neighbors of $u$ in $G^2$ can be partitioned into at most two cliques, if $t \ge 3$. So, we are tempted to consider only squares of 2-degenerate graphs, since for 2-degenerate graphs, for each vertex $u$, 
$|N_h(u)|= t  \le 2$. Unfortunately, even squares of all 2-degenerate graphs are not  2-simplicial.
If we carefully analyze the situation, we can see that  if the two vertices in $N_h(u)$ are adjacent to each other, the square of such a 2-degenerate
graph will be a 2-simplicial graph. This subclass of 2-degenerate graphs is exactly the class of 2-trees.
Note that though any $2$-tree is a $2$-degenerate graph,  the converse is not true. The square of any $2$-tree is a 2-simplicial  graph (but not necessarily a quasi-line graph), but the square of a $2$-degenerate graph may not be a 2-simplicial  graph.  To us, it seems that squares of 2-trees is one of the neatest non-trivial case to consider.  

The main result in this paper is as follows. It shows that Hadwiger's conjecture is true for a special class of 2-simplicial graphs that is not contained in the class of quasi-line graphs. 
The definition of a branch set of a minor will be given at the end of this section.
 
\begin{theorem}\label{thm:2tree}
Hadwiger's conjecture is true for squares of $2$-trees. Moreover, for any $2$-tree $T$, $T^2$ has a clique minor of order $\chi(T^2)$ for which each branch set induces a path. 
\end{theorem}

A graph is called a \emph{generalized $2$-tree} if it can be obtained by allowing one to join a new vertex to a clique of order $1$ or $2$ instead of exactly $2$ in the above-mentioned construction of $2$-trees. (This notion is different from the concept of a partial $2$-tree which is defined as a subgraph of a $2$-tree.) The class of generalized $2$-trees contains all $2$-trees as a proper subclass. The following corollary is implied by (and equivalent to) Theorem~\ref{thm:2tree}.

\begin{corollary}
\label{coro:2tree}
Hadwiger's conjecture is true for squares of generalized $2$-trees. Moreover, for any generalized $2$-tree $G$, $G^2$ has a clique minor of order $\chi(G^2)$ for which each branch set induces a path.
\end{corollary}

In general, in proving Hadwiger's conjecture it is interesting to study the structure of the branch sets forming a clique minor of order no less than the chromatic number. Theorem \ref{thm:2tree} and Corollary \ref{coro:2tree} provide this kind of information for squares of $2$-trees and generalized $2$-trees respectively.

We remark that it is often challenging to establish Hadwiger's conjecture for squares of even very special classes of graphs. We elaborate this point for a few graph classes. 
Obviously, planar graphs form a super class of the class of $2$-trees, but their squares seem to be much more difficult to handle than squares of $2$-trees. In fact, the chromatic number of squares of planar graphs is a very well studied topic in the context of Wegner's conjecture \cite{W}; we will say more about this in section \ref{sec:rem}. Another graph class related to $2$-trees is the class of squares of 2-degenerate graphs. Recently, the first author of this paper and his collaborators \cite{Coll} attempted Hadwiger's conjecture for squares of a special class of $2$-degenerate graphs, namely subdivision graphs. The \emph{subdivision} of a graph $G$, denoted by $S(G)$, is obtained from $G$ by replacing each edge by a path of length two. The square $S(G)^2$ of $S(G)$ is known as the total graph of $G$, and the chromatic number $\chi(S(G)^2)$ is simply the total chromatic number of $G$. Thus, unsurprisingly, Hadwiger's conjecture for squares of subdivision graphs is closely related to the long-standing total coloring conjecture, which can be stated as $\chi(S(G)^2) \le \Delta(G) + 2$, where $\Delta(G)$ is the maximum degree of $G$. It was shown in \cite{Coll} that Hadwiger's conjecture for squares of subdivisions is not difficult to prove if we assume that the total coloring conjecture is true. The best result to date for the total coloring conjecture, obtained by Reed and Molloy \cite{ReedMolloy}, asserts that $\chi(S(G)^2) \le \Delta(G) + {10}^{26}$. Using this result, it was proved in \cite{Coll} that Hadwiger's conjecture is true for squares of subdivisions of highly edge-connected graphs. However, it seems non-trivial to prove Hadwiger's conjecture for squares of subdivisions of all graphs without getting tighter bounds for the total chromatic number.  
        
All graphs considered in the paper are finite, undirected and simple. The vertex and edge sets of a graph $G$ are denoted by $V(G)$ and $E(G)$, respectively. If $u$ and $v$ are adjacent in $G$, then $uv$ denotes the edge joining them. As usual we use $\omega(G)$ to denote the clique number of $G$. A proper coloring of $G$ using exactly $\chi(G)$ colors is called an \emph{optimal coloring} of $G$. 
   
An $H$-minor of a graph $G$ can be thought as a family of $t = |V(H)|$ vertex-disjoint subgraphs $G_1, \ldots, G_t$ of $G$ such that each $G_i$ is connected (possibly $K_1$) and the graph constructed in the following way is isomorphic to $H$: Identify all vertices of each $G_i$ to obtain a single vertex $v_i$, and draw an edge between $v_i$ and $v_j$ if and only if there exists at least one edge of $G$ between $V(G_i)$ and $V(G_j)$. The vertex set of each subgraph $G_i$ in the family is called a {\em branch set} of the minor $H$. This equivalent definition of a minor will be used throughout the paper. 

The proof of Theorem \ref{thm:2tree} is the main body of the paper and will be given in section \ref{sec:2tree}. In section \ref{sec:coro2trees} we prove Corollary \ref{coro:2tree} using Theorem \ref{thm:2tree}, and in the last section we make a few remarks to conclude the paper.

\section{Proof of Theorem \ref{thm:hadsquare}} 
\label{sec:hadsquare}

It suffices to prove that if Hadwiger's conjecture is true for squares of all split graphs then it is also true for all graphs. 

So we assume that Hadwiger's conjecture is true for squares of split graphs. Let $G$ be an arbitrary graph with at least two vertices. Since deleting isolated vertices does not affect the chromatic or Hadwiger number, without loss of generality we may assume that $G$ has no isolated vertices. Construct a split graph $H$ from $G$ as follows: For each vertex $x$ of $G$, introduce a vertex $v_x$ of $H$, and for each edge $e$ of $G$, introduce a vertex $v_e$ of $H$, with the understanding that all these vertices are pairwise distinct. Denote 
$$
S = \{v_x: x \in V(G)\},\;\, C = \{v_e: e \in E(G)\}.
$$ 
Construct $H$ with vertex set $V(H) = S \cup C$ in such a way that no two vertices in $S$ are adjacent, any two vertices in $C$ are adjacent, and $v_x \in S$ is adjacent to $v_e \in C$ if and only if $x$ and $e$ are incident in $G$. Obviously, $H$ is a split graph as its vertex set can be partitioned into the independent set $S$ and the clique $C$. 

\smallskip
\textit{Claim 1:} The subgraph of $H^2$ induced by $S$ is isomorphic to $G$. 

In fact, for distinct $x, y \in V(G)$, $v_x$ and $v_y$ are adjacent in $H^2$ if and only if they have a common neighbor in $H$. Clearly, this common neighbor has to be from $C$, say $v_e$ for some $e \in E(G)$, but this happens if and only if $x$ and $y$ are adjacent in $G$ and $e=xy$. Therefore, $v_x$ and $v_y$ are adjacent
in $H^2$ if and only if $x$ and $y$ are adjacent in $G$. This proves Claim 1.

\smallskip
\textit{Claim 2:} In $H^2$ every vertex of $S$ is adjacent to every vertex of $C$.

This follows from the fact that $C$ is a clique of $H$ and $x$ is incident with at least one edge in $G$.

\smallskip
\textit{Claim 3:} $\chi(H^2) = \chi(G) + |C|$. 

In fact, by Claim 1 we may color the vertices of $S$ with $\chi(G)$ colors by using an optimal coloring of $G$ (that is, choose an optimal coloring $\phi$ of $G$ and assign the color $\phi(x)$ to $v_x$ for each $x \in V(G)$). We then color the vertices of $C$ with $|C|$ other colors, one for each vertex of $C$. It is evident that this is a proper coloring of $H^2$ and hence $\chi(H^2) \le \chi(G) + |C|$. On the other hand, since $C$ is a clique, it requires $|C|$ distinct colors in any proper coloring of $H^2$. Also, by Claim 2 none of these $|C|$ colors can be assigned to any vertex of $S$ in any proper coloring of $H^2$, and by Claim 1 the vertices of $S$ need at least $\chi(G)$ colors in any proper coloring of $H^2$. Therefore, $\chi(H^2) \ge \chi(G) + |C|$ and Claim 3 is proved.

\smallskip
\textit{Claim 4:} $\eta(H^2) = \eta(G) + |C|$.
 
To prove this claim, consider the branch sets of $G$ that form a clique minor of $G$ with order $\eta(G)$, and take the corresponding branch sets in the subgraph of $H^2$ induced by $S$. Take each vertex of $C$ as a separate branch set. Clearly, these branch sets produce a clique minor of $H^2$ with order $\eta(G) + |C|$. Hence $\eta(H^2) \ge \eta(G) + |C|$.  

To complete the proof of Claim 4, consider an arbitrary clique minor of $H^2$, say, with branch sets $B_1,B_2,\ldots,B_k$. Define $B_i' = B_i$ if $B_i \cap C = \emptyset$ (that is, $B_i \subseteq S$) and $B_i' = B_i \cap C$ if $B_i \cap C \ne \emptyset$. It can be verified that $B_1',B_2',\ldots,B_k'$ also produce a clique minor of $H^2$ with order $k$. Thus, if $k > \eta(G) + |C|$, then there are more than $\eta(G)$ branch sets among $B_1',B_2',\ldots,B_k'$ that are contained in $S$. In view of Claim 1, this means that $G$ has a clique minor of order strictly bigger than $\eta(G)$, contradicting the definition of $\eta(G)$. Therefore, any clique minor of $H^2$ must have order at most $\eta(G) + |C|$ and the proof of Claim 4 is complete. 
	
Since we assume that Hadwiger's conjecture is true for squares of split graphs, we have $\eta(H^2) \ge \chi(H^2)$. This together with Claims 3-4 implies $\eta(G) \ge \chi(G)$; that is, Hadwiger's conjecture is true for $G$. This completes the proof of Theorem \ref{thm:hadsquare}.

\section{Proof of Theorem \ref{thm:2tree}} 
\label{sec:2tree}
 
\newcommand{\lvl}{\emph{level }}
\newcommand{\lmax}{\ell_{max}}
\newcommand{\childn}{\emph{vertex-children }}
\newcommand{\child}{\emph{vertex-child }}
\newcommand{\echildn}{\emph{edge-children }}
\newcommand{\echild}{\emph{edge-child }}
\newcommand{\edesc}{\emph{edge-descendant }}
\newcommand{\desc}{\emph{vertex-descendant }}
\newcommand{\col}{\chi({G^2})}
\newcommand{\cqm}{\eta({G^2})}
\newcommand{\cq}{\omega(G^2)}
\newcommand{\piv}{\emph{pivot }}
\newcommand{\np}{N_2[p] }
\newcommand{\n}[1]{N_2[#1]}
\newcommand{\coup}{\mathtt{mate}}
\newcommand{\brset}{\mathcal{B}}

\subsection{Prelude}
\label{subsec:pre}

By the definition of a $k$-tree given in the previous section, a 2-tree is a graph that can be recursively constructed by applying the following operation a finite number of times beginning with $K_2$: Pick an edge $e=uv$ in the current graph, introduce a new vertex $w$, and add edges $uw$ and $vw$ to the graph. We say that $e$ is \emph{processed} in this \emph{step} of the construction. We also say that $w$ is a \emph{vertex-child} of $e$; each of $uw$ and $vw$ is an \emph{edge-child} of $e$; $e$ is the \emph{parent} of each of $w, uw$ and $vw$; and $uw$ and $vw$ are \emph{siblings} of each other. An edge $e_2$ is said to be an \emph{edge-descendant} of an edge $e_1$, if either $e_2 = e_1$, or recursively, the parent of $e_2$ is an \emph{edge-descendant} of $e_1$. A vertex $v$ is said to be a \emph{vertex-descendant} of an edge $e$ if $v$ is a \emph{vertex-child} of an \emph{edge-descendant} of $e$.  

An edge $e$ may be processed in more than one step. If necessary, we can change the order of edge-processing so that $e$ is processed in consecutive steps but the same $2$-tree is obtained. So without loss of generality we may assume that for each edge $e$ all the steps in which $e$ is processed occur consecutively.

We now define a \emph{level} for each edge and each vertex as follows. Initially, the level of the first edge and its end-vertices is defined to be $0$. Inductively, any vertex-child or edge-child of an edge with level $k$ is said to have level $k+1$. Observe that two edges that are siblings of each other have the same level. If there exists a pair of edges $e, f$ with levels $i, j$ respectively such that $i < j$ and the batch of consecutive steps where $e$ is processed is immediately after the batch of consecutive steps where $f$ is processed, then we can move the batch of steps where $e$ is processed to the position immediately before the processing of $f$ without changing the structure of the $2$-tree. We repeat this procedure until no such pair of edges exists. So without loss of generality we may assume that a \emph{breadth-first ordering} is used when processing edges, that is, edges of level $i$ are processed before edges of level $j$ whenever $i<j$.    

To prove Theorem~\ref{thm:2tree}, we will prove $\eta(T^2) \ge \chi(T^2)$ for any $2$-tree $T$. In the simplest case where $\chi(T^2)=2$, this inequality is true as $T^2$ has at least one edge and so contains a $K_2$-minor. Moreover, in this case both branch sets of this $K_2$-minor are singletons (and so induce paths of length $0$). 

In what follows $T$ is an arbitrary 2-tree with $\chi(T^2) \ge 3$. Denote by $T_i$ the $2$-tree obtained after the $i^{\textrm{th}}$ step in the construction of $T$ as described above. Then there is a unique positive integer $\ell$ such that $\chi(T^2)=\chi(T^2_{\ell})=\chi(T^2_{{\ell}-1})+1$. Define
$$
G=T_{\ell}.
$$
We will prove that $\cqm \ge \col$ and $G^2$ has a clique minor of order $\col$ for which each branch set induces a path. Once this is achieved, we then have $\eta(T^2) \ge \eta(G^2) \ge \chi(G^2) = \chi(T^2)$ and $T^2$ contains a clique minor of order $\chi(T^2)$ whose branch sets induce paths, as required to complete the proof of Theorem~\ref{thm:2tree}. 

Given $X \subseteq V(G)$, define
$$
N(X) = \{v \in V(G) \setminus X: \text{$v$ is adjacent in $G$ to at least one vertex in $X$}\}.
$$ 
Define 
$$
N[X] = N(X) \cup X,\;\, N_2[X]=N[N[X]],\;\, N_2(X)=N_2[X]\setminus X.
$$
In particular, for $x \in V(G)$, we write $N(x)$, $N[x]$, $N_2(x)$, $N_2[x]$ in place of $N(\{x\})$, $N[\{x\}]$, $N_2(\{x\})$, $N_2[\{x\}]$, respectively. 

Denote by $\lmax$ the maximum level of any edge of $G$. Then the maximum level of any vertex in $G$ is also $\lmax$. Observe that the level of the last edge processed is $\lmax-1$, and none of the edges with level $\lmax$ has been processed at the completion of the $\ell^{\text{th}}$ step, due to the breadth-first ordering of processing edges. Obviously, $\lmax \le \ell$.  

If $\lmax = 0$ or $1$, then $G^2$ is a complete graph and so $\chi(G^2) = \omega(G^2) = \eta(G^2)$. Moreover, $G^2$ contains a clique minor of order $\chi(G^2)$ for which each branch set induces a path of length $0$. Hence the result is true when $\lmax = 0$ or $1$. 

We assume $\lmax \ge 2$ in the rest of the proof. We will prove a series of lemmas that will be used in the proof of Theorem \ref{thm:2tree}. See Figure \ref{fig:relations} for relations among some of these lemmas.

\begin{figure}
	\centering
	\includegraphics[scale=0.35]{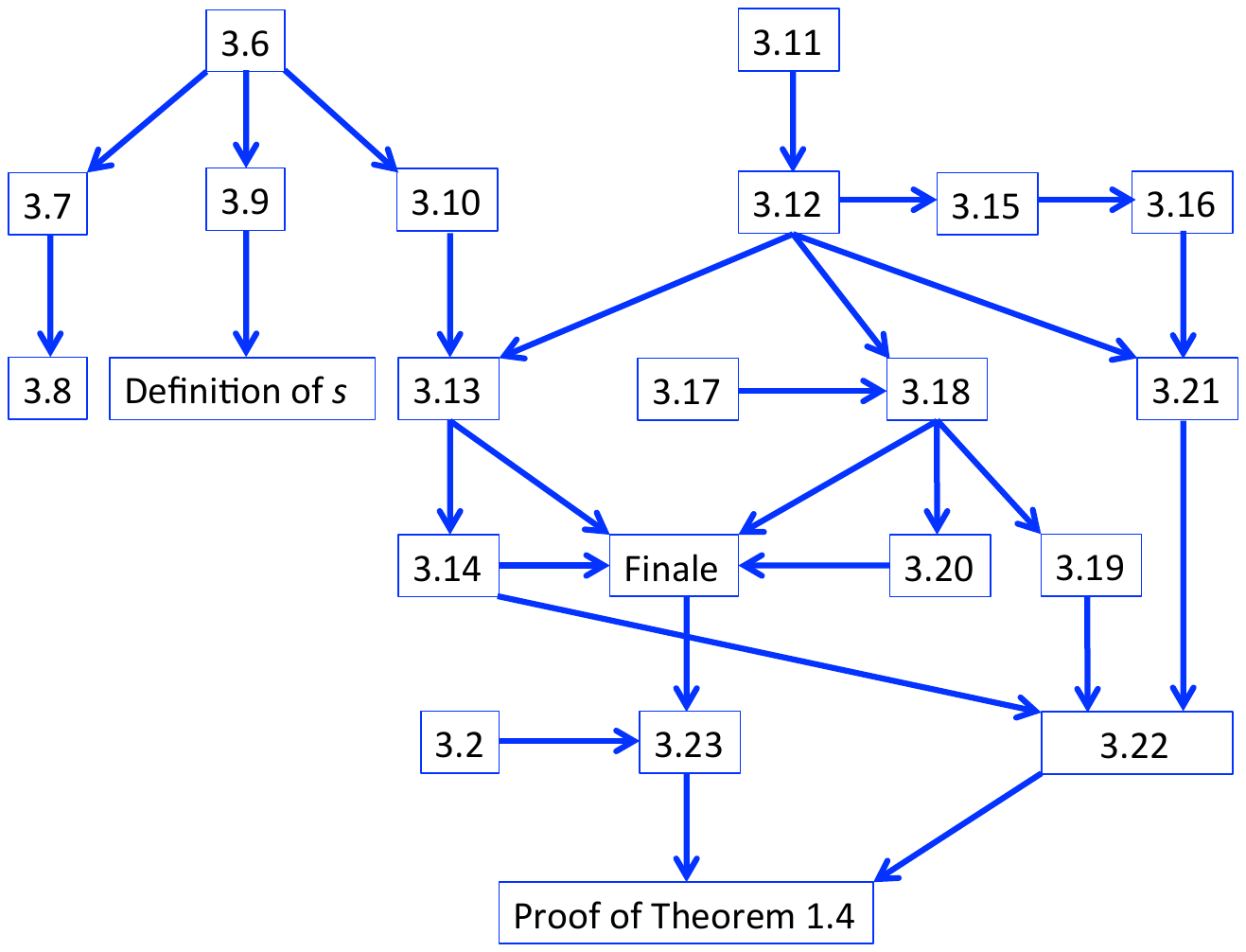}
	\vspace{-1cm}
	\caption{Lemmas to be proved and their relations.}
	\label{fig:relations}
\end{figure}

\subsection{Pivot coloring, pivot vertex and its proximity} 

\begin{lemma}\label{pivot}
There exist an optimal coloring $\mu$ of $G^2$ and a vertex $p$ of $G$ at level $\lmax$ such that $p$ is the only vertex with color $\mu(p)$.
\end{lemma}

\begin{proof}
Let $v$ be the vertex introduced in the step $\ell$. Then $v$ has level $\lmax$. By the definition of $G = T_{\ell}$, there exists a proper coloring of $T^2_{\ell-1}$ using $\chi(G^2)-1$ colors. Extend this coloring to $G^2$ by assigning a new color to $v$. This extended coloring $\phi$ is an optimal coloring of $G^2$ under which $v$ is the only vertex with color $\phi(v)$. 
\end{proof}

Note that, apart from the pair $(\phi, v)$ in the proof above, there may be other pairs $(\mu, p)$ with the property in Lemma \ref{pivot}. 

In the remaining part of the paper, we will use the following notation and terminology (see Figure~\ref{fig:mainfig} for an illustration):  

\begin{itemize}
\item $\mu, p$: an optimal coloring of $G^2$ and a vertex of $G$, respectively, as given in Lemma \ref{pivot}; we fix a pair $(\mu, p)$ such that the minimum level of the vertices in $N(p)$ is as large as possible; we call this particular $\mu$ the \emph{pivot coloring} and this particular $p$ the \emph{pivot vertex};

\item $uw$: the parent of $p$;

\item $t$: the vertex such that $w$ is a child of $ut$, so that the level of $ut$ is $\lmax-2$, and $uw$ and $wt$ are siblings with level $\lmax-1$ (the existence of $t$ is ensured by the fact $\lmax \ge 2$);

\item $B$: the set of vertex-children of $wt$;  

\item $C$: the set of vertex-children of $uw$; 

\item $\mu(X) = \{\mu(x): x \in X\}$, for any subset $X \subseteq V(G)$;

\item when we say the color of a vertex, we mean the color of the vertex under the coloring $\mu$, unless stated otherwise.  
\end{itemize}

\begin{lemma}\label{lemma:allcolors}
All colors used by $\mu$ are present in $N_2[p]$.
\end{lemma}
\begin{proof}
If there is a color $c$ used by $\mu$ that is not present in $N_2\left[ p \right]$, then we can re-color $p$ with $c$. Since $p$ is the only vertex with color $\mu(p)$ under $\mu$, we then obtain a proper coloring of $G^2$ with $\chi(G^2)-1$ colors, which is a contradiction.
\end{proof}


\begin{figure}
	\centering
	\includegraphics[scale=0.35]{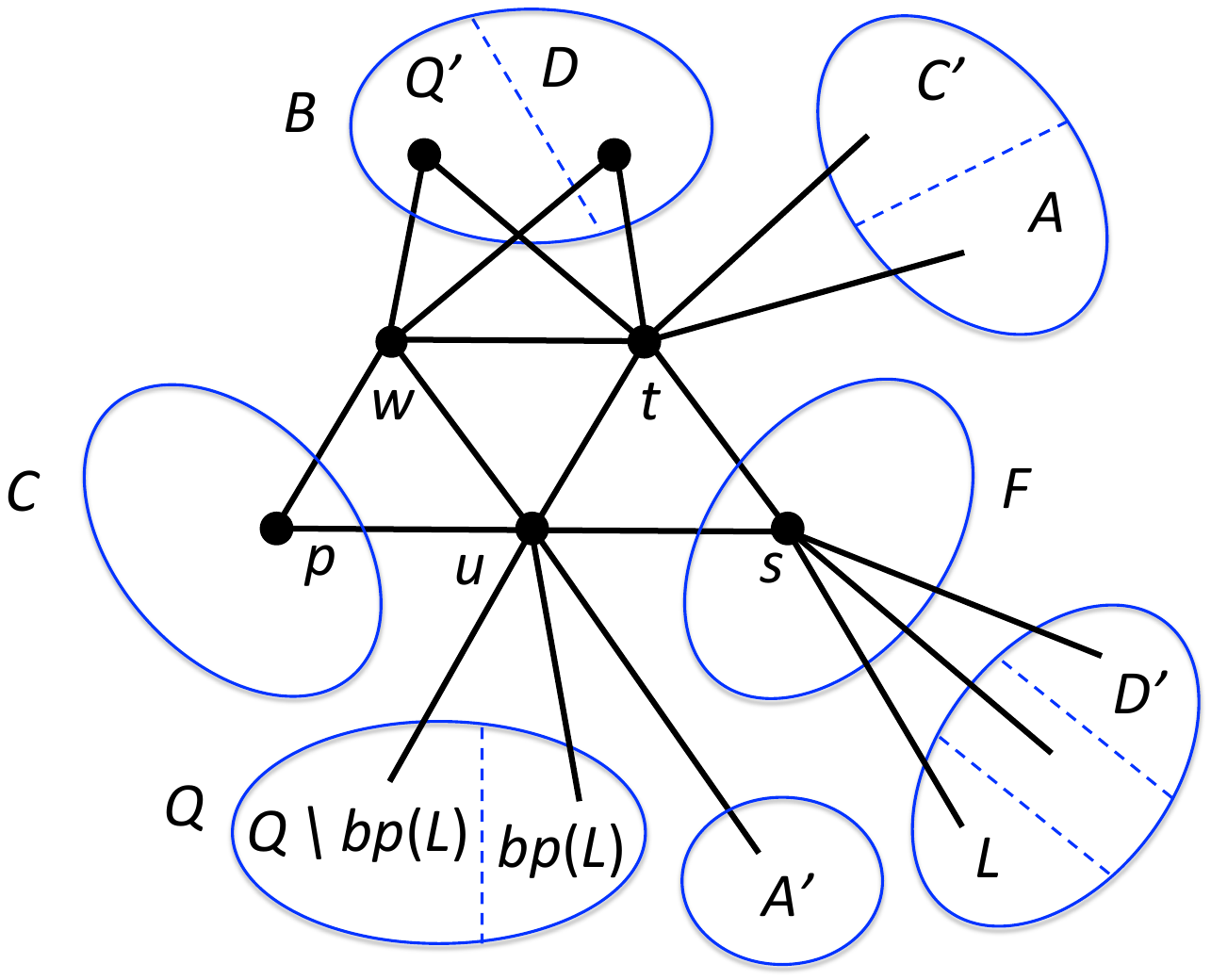}
        	\caption{Vertex subsets of the 2-tree $G$ used in the proof of Theorem \ref{thm:2tree}.}
	\label{fig:mainfig}
\end{figure}

\begin{lemma}\label{unprocessed}
$N(b)=\{w,t\}$ for any $b\in B$, and $N(c)=\{u,w\}$ for any $c\in C$.
\end{lemma}
\begin{proof}
Since both $bw$ and $bt$ have level $\lmax$, they have not been processed at the completion of the $\ell^{\text{th}}$ step. Hence the first statement is true. The second statement can be proved similarly.
\end{proof}

Define
$$
F = (N(u)\cap N(t)) \setminus \{w\}
$$
$$
C' = \{x\in N(t): \mu(x) \in \mu(C)\}
$$ 
$$
A = N(t)\setminus (B \cup F \cup C' \cup \{u, w\}).
$$ 
Note that $C' \subseteq N(t)\setminus (B \cup F \cup \{u, w\})$ and $\{A, C'\}$ is a partition of $N(t)\setminus (B \cup F \cup \{u, w\})$. Note also that there may exist edges between $F$ and $A \cup C'$.

\begin{lemma}\label{lemma:AA}
	$\mu(A)\subseteq \mu(N(u)\setminus (C \cup F \cup \{w,t\}))$.
\end{lemma}
\begin{proof}
Let $a\in A$. Clearly, $\mu(a) \notin \mu(N_2(a))$. On the other hand, $\mu(a) \in \mu(N_2\left[ p \right])$ by Lemma~\ref{lemma:allcolors}. So $\mu(a) \in \mu (N_2[p]\setminus N_2(a))$. Since $N_2[p]\setminus N_2(a)\subseteq N(u)\setminus (F \cup \{w, t\})$, it follows that $\mu(a) \in \mu(N(u) \setminus (F \cup \{w,t\}))$. By the definition of $A$, we also have $\mu(a) \notin \mu(C)$. Therefore, $\mu(a) \in \mu(N(u)\setminus (C \cup F \cup \{w,t\}))$. 
\end{proof}

By Lemma~\ref{lemma:AA}, for each color $c\in \mu(A)$, there is a $c$-colored vertex in $N(u)\setminus (C \cup F \cup \{w,t\})$. On the other hand, no two vertices in $N(u)$ can have the same color. So each color in $\mu(A)$ is used by exactly one vertex in $N(u)$. Let 
$$
A' = \{x \in N(u): \mu(x) \in \mu(A)\}.
$$ 
Then $A' \subseteq N(u)\setminus (C \cup F \cup \{w,t\})$ and
$$
\mu(A')=\mu(A).
$$ 
Since no two vertices in $A$ ($A'$, respectively) are colored the same, the relation $\mu(a) = \mu(a')$ defines a bijection $a \mapsto a'$ from $A$ to $A'$. We call $a$ and $a'$ the mates of each other and denote the relation by 
$$
a = \coup(a'),\;\, a' = \coup(a).
$$
Note that $a \ne a'$ as $A$ and $A'$ are disjoint. Define
$$
Q=N(u)\setminus (A' \cup C \cup F \cup \{ w,t\}).
$$
Then $\{A', Q\}$  is a partition of $N(u)\setminus (C \cup F \cup \{w,t\})$. Note that there may exist edges between $F$ and $A' \cup Q$.

Define
$$
D = \{x\in B: \mu(x)\notin \mu(N(u))\}
$$ 
$$
Q' = B \setminus D.
$$ 
Then $A',A,C,C',D,F,Q,Q',\{u,w,t\}$ are pairwise disjoint. See Figure~\ref{fig:mainfig} for an illustration of these sets.

\subsection{A few lemmas}

\begin{lemma}\label{lemma:empty}
Suppose $D=\emptyset$. Then $\cqm \ge \col$. Moreover, $\col = \omega(G^2)$ and so $G^2$ contains a clique minor of order $\col$ for which each branch set is a singleton. 
\end{lemma}
\begin{proof}
Since $D=\emptyset$, we have $N_2[p]=N[u]\cup Q'$. So by Lemma~\ref{lemma:allcolors} all colors of $\mu$ are present in $N[u]\cup Q'$. However, $\mu(Q')\subseteq \mu(N[u])$ by the definition of $Q'$. So all colors of $\mu$ are present in $N[u]$. Since $N[u]$ is a clique of $G^2$, it follows that $\col = |N[u]| \le \cq$. Therefore, $\col=\cq\le \cqm$.
\end{proof}

\begin{lemma}\label{lemma:duniquecolor}
For any $d\in D$, no vertex in $N_2[p]$ other than $d$ is colored $\mu(d)$.
\end{lemma}
\begin{proof}
Suppose that there is a vertex in $N_2[p] \setminus \{d\}$ with color $\mu(d)$.
Such a vertex must be in $N_2[p]\setminus N_2[d]$. However, $N_2[p]\setminus N_2[d]=Q\cup A'$, but 
$\mu(d)\notin \mu(Q)$ by the definition of $D$ and $\mu(d)\notin \mu(A')=\mu(A)$ as $A\subseteq N_2[d]$. This contradiction proves the result. 
\end{proof}

\begin{lemma}\label{lemma:QQ}
Suppose $D \ne \emptyset$. Then $\mu(Q)=\mu(Q')$.  
\end{lemma}
\begin{proof}
We prove $\mu(Q')\subseteq \mu(Q)$ first. 
By the definition of $Q'$, $\mu(Q')\subseteq \mu(N(u))$. Clearly, $\mu(Q')\cap \mu(N_2(Q'))=\emptyset$, and $\mu(Q')\cap \mu(A')=\emptyset$ as $\mu(A')=\mu(A)$. Hence $\mu(Q') \subseteq \mu(N(u) \setminus (N_2(Q') \cup A'))$. However, $N(u)\setminus (N_2(Q') \cup A') = Q$. Therefore, $\mu(Q')\subseteq \mu(Q)$.
	
	Now we prove $\mu(Q)\subseteq \mu(Q')$.
	Suppose otherwise. Say, $q \in Q$ satisfies $\mu(q) \notin \mu(Q')$. Since $D \ne \emptyset$ by our assumption, we may take a vertex $d \in D$. 
	We claim that $\mu(q)\notin N_2(d)$.
	This is because $N_2(d)\setminus N_2[q]\subseteq A \cup C' \cup Q' \cup D $, but  $\mu(q)\notin \mu(A)=\mu(A')$, $\mu(q)\notin \mu(C')\subseteq \mu(C)$, $\mu(q)\notin \mu(Q')$, and $\mu(q)\notin \mu(D)$ by the definition of $D$. So we can recolor $d$ with $\mu(q)$. By Lemma~\ref{lemma:duniquecolor}, we can then recolor $p$ with $\mu(d)$. In this way we obtain a proper coloring of $G^2$ with $\chi(G^2)-1$ colors, which is a contradiction. Hence $\mu(Q)\subseteq \mu(Q')$.
\end{proof}

\begin{lemma}
\label{lemma:Aempty}
Suppose $D \ne \emptyset$ but $A=\emptyset$. Then $\cqm \ge \col$. Moreover, $\col = \omega(G^2)$ and so $G^2$ contains a clique minor of order $\col$ for which each branch set is a singleton. 
\end{lemma}

\begin{proof}
Since $A=\emptyset$, we have $A'=\emptyset$ and $\mu(N_2[p])=\mu(N[w]\cup F)$ by Lemma \ref{lemma:QQ}. By Lemma \ref{lemma:allcolors}, $|\mu(N_2[p])|=\chi(G^2)$. On the other hand, $N[w]\cup F$ is a clique of $G^2$ and so $|\mu(N[w] \cup F)| \le \omega(G^2)$. So $\chi(G^2)=|\mu(N_2[p])|=|\mu(N[w]\cup F)| \le \omega(G^2)$, and therefore $\chi(G^2) = \omega(G^2) \le\eta(G^2)$.
\end{proof}

Due to Lemmas \ref{lemma:empty} and \ref{lemma:Aempty}, in the rest of the proof we assume without mentioning explicitly that $D \ne \emptyset$ and $A \ne \emptyset$. Then $A' \ne \emptyset$ and $\mu(Q)=\mu(Q')$.
 
\begin{lemma}
\label{lmax2}
The following hold:
\begin{itemize}
\item[\rm (a)] $\lmax \ge 3$;
\item[\rm (b)] the level of $u$ is $\lmax-2$.
\end{itemize}
\end{lemma}
\begin{proof}
(a) We have assumed $\lmax \ge 2$. Suppose $\lmax=2$ for the sake of contradiction. Take $a'\in A'$ and $d\in D$. Since $\lmax=2$, we have that $ut$ is the only edge with level $0$, and moreover $V(G)=N[\{u,t\}]$. 

We claim that no vertex in $N_2[a']$ is colored $\mu(d)$ under the coloring $\mu$. Suppose otherwise. Say, $d_1$ is such a vertex. Then $d_1\neq d$ as $d \in D$ but $D \cap N_2[a'] = \emptyset$. We have $d_1\notin N(u)$ by the definition of $D$. We also have $d_1 \notin N(t)$ for otherwise two distinct vertices in $N(t)$ have the same color. Thus, $d_1 \notin N(u) \cup N(t) = N[\{u,t\}] = V(G)$, a contradiction. Therefore, no vertex in $N_2[a']$ is colored $\mu(d)$.

So we can recolor $a'$ with color $\mu(d)$ but retain the colors of all other vertices. In this way we obtain another proper coloring of $G^2$. Observe that $a'$ was the only vertex in $N_2[p]$ with color $\mu(a')$ under $\mu$ as $N_2[p] \subseteq \n{a'} \cup N_2(a)$, where $a = \coup(a') \notin N_2[p]$. Since $a'$ has been recolored $\mu(d)$, we can recolor $p$ with $\mu(a')$ to obtain a proper coloring of $G^2$ using fewer colors than $\mu$, but this contradicts the optimality of $\mu$. 

(b) Suppose otherwise. Since the level of $ut$ is $\lmax-2$, the level of $t$ must be $\lmax-2$ and the level of $u$ must be smaller than $\lmax-2$. Take $d \in D$. Denote by $\mu'$ the coloring obtained by exchanging the colors of $d$ and $p$ (while keeping the colors of all other vertices). By Lemma \ref{lemma:duniquecolor}, $\mu'$ is a proper coloring of $G^2$. Note that $d$ is the only vertex with color $\mu'(d) = \mu(p)$ under the coloring $\mu'$. The minimum level of a vertex in $N(d)$ is $\lmax-2$, and the minimum level of a vertex in $N(p)$ is smaller than $\lmax-2$ since the level of $u$ is smaller than $\lmax-2$. However, this means that we would have selected respectively $\mu'$ and $d$ as the pivot coloring and pivot vertex instead of $\mu$ and $p$, which is a contradiction.
\end{proof}

In the sequel we fix a vertex $s \in F$ such that $ut$ is a child of $st$. The existence of $s$ is ensured by Lemma \ref{lmax2}. Note that the level of $st$ is $l_{max}-3$, and $us$ is the sibling of $ut$ and has level $\lmax-2$.  

\subsection{Bichromatic paths}

\begin{definition}
\emph{Given a proper coloring $\phi$ of $G^2$ and two distinct colors $r$ and $g$, a path in $G^2$ is called a \emph{$(\phi,r,g)$-bichromatic path} if its vertices are colored $r$ or $g$ under the coloring $\phi$.}
\end{definition}
\begin{lemma}
	\label{lemma:pathD}
For any $a'\in A'$ and $d\in D$, there exists a $(\mu,\mu(a'),\mu(d))$-bichromatic path from $a'$ to $\coup(a')$ in $G^2$. 
\end{lemma}
\begin{proof}
	We will use the well known Kempe chain technique.
Let $a=\coup(a')$. Denote $r = \mu(a')$ ($=\mu(a)$) and $g = \mu(d)$. Then $r \ne g$ as $d \in N_2(a)$. Consider the subgraph $H$ of $G^2$ induced by the set of vertices with colors $r$ and $g$ under $\mu$. Let $H'$ be the connected component  of $H$ containing $a'$. It suffices to show that $a$ is contained in $H'$.

Suppose to the contrary that $a \not \in V(H')$. Define
$$
\mu'(v) = \left\{ 
\begin{array}{ll}
\mu(v), & \mbox{if } v\in V(G)\setminus (V(H')\cup \{p\}) \\
r, & \mbox{if } v = p \\
r, & \mbox{if $v\in V(H')$ and $\mu(v)=g$} \\
g, & \mbox{if $v\in V(H')$ and $\mu(v)=r$.}  
\end{array}
\right.
$$
In particular, $\mu'(a')=g$. We will prove that $\mu'$ is a proper coloring of $G^2$, which will be a contradiction as $\mu'$ uses less colors than $\mu$. Since exchanging colors $r$ and $g$ within $H'$ does not produce an improper coloring, in order to prove $\mu'$ is proper, it suffices to prove that $N_2(p)$ does not contain any vertex with color $\mu'(p)$ under $\mu'$. Suppose otherwise. Say, $v\in N_2(p)$ satisfies $\mu'(v)=\mu'(p)=r$. Consider first the case when $v\in V(H')$. In this case, we have $\mu(v)=g$, and so $v=d$ since by Lemma~\ref{lemma:duniquecolor}, $d$ is the only vertex in $N_2[p]$ with color $g$ under $\mu$. On the other hand, $d\notin V(H')$ as $a\notin V(H')$ is the only vertex in $N_2[d]$ with color $r$ under $\mu$. Hence $v\notin V(H')$, which is a contradiction. Now consider the case when $v\notin V(H')$. In this case, we have $\mu(v)=r$. Since $N_2[p]\subseteq \n{a'}\cup N_2(a)$, $a'$ is the only vertex in $N_2[p]$ with color $r$ under $\mu$. So $v=a'\in V(H')$, which is again a contradiction.
\end{proof}

\begin{lemma}\label{lemma:l-1}
For any edge $e=xy$ with level $\lmax -2$ and any vertex-descendant $z$ of $e$, we have $N_2(z)\subseteq N[\{x, y\}]$.
\end{lemma}

\begin{proof}
Consider an arbitrary vertex $v$ in $N_2(z)$. Since the level of $e$ is $\lmax-2$, there are only two possibilities for $z$. The first possibility is that $z$ is a vertex-child of $e$. In this possibility, either $v$ is a vertex-child of $x z$ or $yz$, or $v \in \{x, y\}$,  or $v \in N(x) \cup N(y)$; in each case we have $v \in N[\{x, y\}]$. The second possibility is that $z$ is the vertex-child of an edge-child of $e$. Without loss of generality we may assume that $z$ is the vertex-child of $x q$, where $q$ is a vertex-child of $e$. Then either $v$ is a vertex-child of $yq$ or $v \in N[x]$; in each case we have $v \in N[\{x, y\}]$.
\end{proof}

\begin{lemma}
\label{lemma:n2adjacentuts}
The following hold:
\begin{itemize}
\item[\rm (a)] $N_2(A'\cup Q)\subseteq N[\{u,t,s\}]$;
\item[\rm (b)] if $v\in N_2(A'\cup Q)$ and $\mu(v)\in \mu(B)$, then $v\in N(\{u,s\})$;
\item[\rm (c)] if $v\in N_2(A'\cup Q)$ and $\mu(v)\in \mu(D)$, then $v\in N(s)$.
\end{itemize}
\end{lemma}

\begin{proof}
(a) Any vertex $x\in A'\cup Q$ is a vertex-descendant of $ut$ or $us$. Since the levels of $ut$ and $us$ are both $\lmax-2$, by Lemma~\ref{lemma:l-1}, if $x$ is a vertex-descendant of $ut$ then $N_2(x)\subseteq N[\{u,t\}]$, and if $x$ is a vertex-descendant of $us$ then $N_2(x)\subseteq N[\{u,s\}]$. Therefore, $N_2(x)\subseteq N[\{ u,t,s\}]$.

(b) Consider $v\in N_2(x)$ for some $x\in A' \cup Q$ such that $\mu(v)\in \mu(B)$. Since $v\in N[\{u,t,s\}]$ by (a),  
it suffices to prove $v\notin N[t]$. Suppose otherwise. Since $\mu(v)\in \mu(B)$, if $v \not \in B$, then both $v \in N[t]$ and another neighbor of $t$ in $B$ have color $\mu(v)$, a contradiction. Hence $v\in B$. Since $N_2(x)\cap B=\emptyset$, we then have $v\notin N_2(x)$, but this is a contradiction.

(c) By (b), every vertex $v\in N_2(A'\cup Q)$ with $\mu(v)\in \mu(D)$ must be in $N(\{u,s\})$. If $v \in N(u)$, then $\mu(v) \in \mu(N(u))$ and so $\mu(v) \notin \mu(D)$ by the definition of $D$, a contradiction. Hence $v \notin N(u)$ and therefore $v\in N(s)$.
\end{proof}

Define 
$$
D' = \{x\in N(s): \mu(x)\in\mu(D)\}.
$$

\begin{lemma}
\label{lemma:pathd-adj}
The following hold:
	\begin{itemize}
		\item[\rm (a)] $\mu(D')=\mu\left( D \right)$;
		\item[\rm (b)] for any $a'\in A'$ and $d'\in D'$, there exists a $(\mu,\mu(a'),\mu(d'))$-bichromatic path in $G^2$ from $a'$ to $\coup(a')$ such that $d'$ is adjacent to $a'$ in this path. 
	\end{itemize}
\end{lemma}
\begin{proof}
	Let $d$ be an arbitrary vertex in $D$.	
	Let $a_1'$ and $a_2'$ be arbitrary vertices in $A'$.
	By Lemma~\ref{lemma:pathD}, there exists a $(\mu,\mu(a'_1),\mu(d))$-bichromatic path $P_1$ from $a_1'$ to $\coup(a_1')$, and there exists a $(\mu,\mu(a'_2),\mu(d))$-bichromatic path $P_2$ from $a_2'$ to $\coup(a_2')$. Note that $P_1$ and $P_2$ each has at least three vertices.
	Let $d_1$  be the vertex adjacent to $a_1'$ in $P_1$ and $d_2$ the vertex adjacent to $a_2'$ in $P_2$.
	Clearly, $\mu(d_1)=\mu(d_2)=\mu(d)$.
	By Lemma~\ref{lemma:n2adjacentuts}(c), both $d_1$ and $d_2$ are in  $N(s)$, and hence $d_1\in N_2[d_2]$. This together with $\mu(d_1)=\mu(d_2)$ implies $d_1=d_2$.
	Thus, for any $d\in D$, there exists $d'\in N(s)$ with $\mu(d')=\mu(d)$ such that for each $a'\in A'$, there exists a $(\mu,\mu(a'),\mu(d))$-bichromatic path from $a'$ to $\coup(a')$ that passes through the edge $a' d'$. Both statements in the lemma easily follow from the statement in the previous sentence.  
\end{proof}

Since no two vertices in $D$ ($D'$, respectively) are colored the same, by Lemma \ref{lemma:pathd-adj} we have $|D| = |D'|$ and every $d' \in D'$ corresponds to a unique $d \in D$ such that $\mu(d) = \mu(d')$, and vice versa. We call $d$ and $d'$ the mates of each other, written $d = \coup(d')$ and $d' = \coup(d)$. Lemma \ref{lemma:pathd-adj} implies the following results (note that for $a'\in A'$ and $d'\in D'$, $\coup(a')$ is adjacent to $\coup(d')$ in $G^2$). 

\begin{corollary}
\label{corollary:dadja'}
The following hold:
\begin{itemize}
\item[\rm (a)] each $a'\in A'$ is adjacent to each $d'\in D'$ in $G^2$;
\item[\rm (b)] for any $a'\in A'$ and $d'\in D'$, there exists a $(\mu,\mu(a'),\mu(d'))$-bichromatic path from $d'$ to $\coup(d')$ in $G^2$.
\end{itemize}	
\end{corollary}

\subsection{Bridging sets, bridging sequences, and re-coloring}

\begin{definition}\label{bridging-set}
{\em An ordered set $\{x_1, x_2, \ldots, x_k\}$ of vertices of $G^2$ is called a \emph{bridging set} if for each $i$, $1\le i\le k$, $x_i \in N(s) \setminus D'$ and there exists a vertex $q_i \in Q$ such that $\mu(q_i)=\mu(x_i)$ and $q_i$ is not adjacent in $G^2$ to at least one vertex in $D'\cup \{x_1,x_2, \ldots, x_{i-1}\}$. Denote $q_i = bp(x_i)$ and call it the \emph{bridging partner} of $x_i$. We also fix one vertex in $D'\cup \{x_1, x_2, \ldots, x_{i-1}\}$ not adjacent to $q_i$ in $G^2$, denote it by $bn(q_i)$, and call it the \emph{bridging non-neighbor} of $q_i$. (If there is more than one candidate, we fix one of them arbitrarily as the bridging non-neighbor.)}
\end{definition}

In the definition above we have $bp(x_i) \neq x_i$ for each $i$, for otherwise $bp(x_i)$ would be adjacent in $G^2$ to all vertices in $N(s)$ and so there is no candidate for the bridging non-neighbor of $bp(x_i)$, contradicting the definition of a bridging set.

In the following we take $L$ to be a fixed bridging set with maximum cardinality.  Note that $\mu(L)\subseteq \mu(Q)$ by the definition of a bridging set.

\begin{definition}
	{\em Given $z \in D' \cup L$, the \emph{bridging sequence} of $z$ is defined as the sequence of distinct vertices $s_1, s_2, \ldots, s_j$ such that $s_1=z$, $s_j\in D'$, and for $2\le i\le j $, $s_i$ is the bridging non-neighbor of the bridging partner of $s_{i-1}$.}	
\end{definition}

By Definition~\ref{bridging-set}, it is evident that the bridging sequence of every $z \in D' \cup L$ exists. In particular, for $d \in D'$, the bridging sequence of $d$ consists of only one vertex, namely $d$ itself.

\begin{lemma}
\label{lemma:bpbn}
Let $x\in L$, $q=bp(x)$ and $y=bn(q)$. If there exists $v\in N_2(q)$ such that $\mu(v)=\mu(y)$,  then $y\in L$ and $v=bp(y)$.
\end{lemma}
\begin{proof}
	We know that $q\in Q$ and $y\in D'\cup L$.
	Since $\mu(L)\subseteq \mu(Q)$ by the definition of a bridging set, we have $\mu(v)=\mu(y)\in \mu(D'\cup L)\subseteq \mu\left( B \right)$.
	Hence, by Lemma~\ref{lemma:n2adjacentuts}(b), $v$ must be in $N(\{s, u\})$. If $v\in N(s)$, then $v=y$, but this cannot happen as $y=bn(q)\notin N_2(q)$. Hence $v\in N(u)$. This implies that $\mu(v)\in \mu(Q\cup A'\cup C \cup \left\{ w,t,s \right\})$ and in particular $\mu(v)\notin \mu(D')$. Therefore, $\mu(y)\notin \mu(D')$, which implies $y\in L$. Since the only vertex in $N(u)$ with color $\mu(y)$ is $bp(y)$, we obtain $v=bp(y)$.  
\end{proof}

\begin{definition}\label{bridging-recoloring}
{\em 
Given a vertex $z\in D' \cup L$ with bridging sequence $s_1,s_2,\ldots, s_j$, define the \emph{bridging re-coloring} $\psi_z$ of $\mu$ with respect to $z$ by the following rules:
	\begin{itemize}
		\item[(a)] $\psi_z(x)=\mu(x)$ for each $x\in V(G) \setminus \left\{ bp(s_i):1\le i< j \right\}$;  
		\item[(b)] \label{constr-step2-bridging-recoloring} 
	  $\psi_z(bp(s_i))=\mu(s_{i+1})$ for $1\le i< j$. 
	\end{itemize}
}
\end{definition}

Observe that for $i \neq j$ we have $\mu(s_i)\neq \mu(s_j)$ as $s_i, s_j \in N(s)$. So each color is used at most once for recoloring in (b) above.

\begin{lemma}\label{psi-optimal}
	For any $z\in D' \cup L$, $\psi_z$ is an optimal coloring of $G^2$.
\end{lemma}
\begin{proof}
Since $\psi_z$ only uses colors of $\mu$, it suffices to prove that it is a proper coloring of $G^2$.
	Let $s_1, s_2, \ldots, s_j$ be the bridging sequence of $z$.
	Suppose to the contrary that $\psi_z$ is not a proper coloring of $G^2$.
	Then by the definition of $\psi_z$ there exists $1 \le i \le j-1$ such that $\psi_z(bp(s_i)) \in \psi_z(N_2(bp(s_i)))$.
	Denote $x=bp(s_i)$. Then there exists $v\in N_2(x)$ such that $\psi_z(v)=\psi_z(x)=\mu(s_{i+1})$. 
	Since $x$ is the only vertex that has the color $\mu(s_{i+1})$ under $\psi_z$ and a different color under $\mu$, we have $\mu(v)=\mu(s_{i+1})$.
	Since $s_{i+1}=bn(x)$, by Lemma~\ref{lemma:bpbn} we have $s_{i+1}\in L$ and $v=bp(s_{i+1})$. Thus $j \ne i+1$.
	However, $\psi_z(bp(s_{i+1}))=\mu(s_{i+2})\neq \mu(s_{i+1})$ by the definition of $\psi_z$.
	Therefore, $\psi_z(v)\neq \mu(s_{i+1})$, which is a contradiction.
\end{proof}

\begin{lemma}\label{rgnotrecolored}
	Let $a'\in A'$, $z \in L$, $r=\mu(a')$ and $g=\mu(z)$. Let $c\in\left\{ r,g \right\}$. Then for any $x \in V(G) \setminus \{bp(z)\}$,
	$\psi_z(x)=c$ if and only if $\mu(x)=c$, whilst $\mu(bp(z)) = \mu(z) = g$ but $\psi_z(bp(z)) \notin \{r, g\}$.  
\end{lemma}

\begin{proof}
	This follows from the definition of $\psi_z$ and the fact that 
	$r, g \notin \mu(\{s_2, s_3, \ldots, s_j\})$ for the bridging sequence $s_1, s_2, \ldots, s_j$ of $z$.
\end{proof}

\begin{lemma}\label{lemma:qpath}
	For any $a'\in A'$ and $q\in L$, there exists a $(\mu, \mu(a'),\mu(q))$-bichromatic path from $a'$ to $\coup(a')$ in $G^2$ which contains the edge $a' q$.
\end{lemma}
\begin{proof}
Denote $\mu(a')=r$, $\mu(q)=g$ and $a=\coup(a')$. In view of Lemma~\ref{rgnotrecolored}, it suffices to prove that there exists a $(\psi_q, r,g)$-bichromatic path from $a'$ to $a$ in $G^2$ which uses the edge $a' q$. Consider the subgraph $H$ of $G^2$ induced by the set of vertices with colors $r$ and $g$ under $\psi_q$. Denote by $H'$ the connected component of $H$ containing $a'$.
	
We first prove that $a \in V(H')$. Suppose otherwise. Define a coloring $\phi$ of $G^2$ as follows: for each $v \in V(H')$, if $\psi_q(v)=r$ then set $\phi(v) = g$, and if $\psi_q(v)=g$ then set $\phi(v) = r$; set $\phi(p) = r$; and set $\phi(x) = \psi_q(x)$ for each $x \in V(G) \setminus (V(H') \cup \{p\})$. 
We claim that $\phi$ is a proper coloring of $G^2$. To prove this it suffices to show $r\notin \phi(N_2(p))$ because exchanging the two colors within $V(H')$ does not produce an improper coloring. Suppose to the contrary that there exists a vertex $v\in N_2(p)$ such that $\phi(v)=r$. If $v\in V(H')$, then $\psi_q(v)=g$ and so $v\neq bp(q)$ by the definition of $\psi_q$. 
Also $\mu(v)=g$ by Lemma~\ref{rgnotrecolored}. 
The only vertices in $N_2[p]$ with color $g$ under $\mu$ are $bp(q)$ and one vertex in $Q'$, say, $q'$. 
Since $v\neq bp(q)$, we have $v=q'$. Since $a\in N_2(q')$, we get $a\in V(H')$, which is a contradiction.
If $v\notin V(H')$, then $\psi_q(v)=r$, and by Lemma~\ref{rgnotrecolored}, $\mu(v)=r$.
However, the only vertex in $N_2[p]$ with color $r$ under $\mu$ is $a'$ (as $N_2[p]\subseteq N_2[\left\{ a,a' \right\}]$, $\mu(a)=\mu(a')=r$ and $a'\notin N_2[p]$). Then $v=a'$ and hence $\psi_q(v)=\psi_q(a')=g\neq r$, which is a contradiction. Thus $\phi$ is a proper coloring of $G^2$. Recall that $p$ is the only vertex in $G$ with color $\mu(p)$ under $\mu$. By the definition of $\psi_q$,  $p$ remains to be the only vertex with color $\mu(p)$ under $\psi_q$. Hence $\phi$ uses one less color than $\psi_q$ as it does not use the color $\psi_q(p)=\mu(p)$. This is a contradiction as by Lemma~\ref{psi-optimal} $\psi_q$ is an optimal coloring of $G^2$.  
Therefore, $a \in V(H')$. 

Since $a \in V(H')$, there is a $(\psi_q, r, g)$-bichromatic path from $a'$ to $a$ in $G^2$. We show that in this path $a'$ has to be adjacent to $q$. 
Suppose otherwise. Say, $v\neq q$ is adjacent to $a'$ in this path. 
Then $\psi_q(v)=g$, and by Lemma~\ref{rgnotrecolored}, $\mu(v)=g$.
By Lemma~\ref{lemma:n2adjacentuts}(b), $v\in N(\{u,s\})$. Since $v\neq q$, we have $v\not \in N(s)$.
Hence, $v\in N(u)$, which implies $v=bp(q)$. 
Since $\psi_q(bp(q))\neq g$ by the definition of $\psi_q$, it follows that $\psi_q(v)\neq g$, but this is a contradiction. This completes the proof. 
\end{proof}

\begin{corollary}\label{qadja'}
Each $a'\in A'$ is adjacent to each $q\in L$ in $G^2$.
\end{corollary}

We now extend the definition of mate to the set $L$.
For each $q\in L$, define $\coup(q)$ to be the vertex in $Q'$ with the same color as $q$ under the coloring $\mu$. 
We now have the following corollary of Lemma~\ref{lemma:qpath}.

\begin{corollary}\label{cor:pathq}
	For any $a'\in A'$ and $q\in L$, there is a $(\mu,\mu(a'),\mu(q))$-bichromatic path from $q$ to $\coup(q)$.
\end{corollary}
\begin{proof}
	This follows because $\coup(a')$ is adjacent to $\coup(q)$ in $G^2$.
\end{proof}

Define 
$$
bp(L) = \{bp(q): q\in L\}.
$$
Then $bp(L)\subseteq Q$, $\mu(bp(L))=\mu(L)$, and $\mu(L \cup (Q \setminus bp(L)))=\mu(Q)=\mu(Q')$.

\begin{lemma}\label{lemma:q3conn-q1d'}
For any $q \in Q \setminus bp(L)$, $D' \cup L \subseteq N_2[q]$.  
\end{lemma}
\begin{proof} 
Suppose otherwise. Say, $q \in Q \setminus bp(L)$ and $z\in (D' \cup L) \setminus N_2[q]$.	
	
Consider first the case when $\mu(q)\in \mu(N(s))$, say, $\mu(q)=\mu(x)$ for some $x \in N(s)$. Then $x \neq q$ for otherwise $z\in N_2[q]$. Also, $x\notin L$ for otherwise, $bp(x)$ and $q$ are adjacent in $G^2$ but have the same color under $\mu$. We also know that $x\notin D'$ as $\mu(D')\cap \mu(N[u])=\emptyset$.  Hence $L\cup \{x\}$ is a larger bridging set than $L$ by taking $bp(x)=q$ and $bn(q)=z$. This contradicts the assumption that $L$ is a bridging set with maximum cardinality.
	
	Henceforth we assume that $\mu(q)\notin \mu(N(s))$. Since $A' \ne \emptyset$ by our assumption, we can take a vertex $a' \in A'$.
	Define a coloring $\phi$ of $G^2$ as follows: set $\phi(q) = \psi_z(z)=\mu(z)$, $\phi(a') = \psi_z(q)=\mu(q)$ and $\phi(p) = \psi_z(a')=\mu(a')$, and color all vertices in $V(G) \setminus \{q, a', p\}$ in the same way as in $\psi_z$. Clearly, $\phi$ uses less colors than $\psi_z$ as it does not use the color $\psi_z(p)$. Since by Lemma~\ref{psi-optimal}, $\psi_z$ is an optimal coloring of $G^2$, $\phi$ cannot be a proper coloring of $G^2$. Hence one of the following three cases must happen. In each case, we will obtain a contradiction and thus complete the proof. Note that, by the definition of $Q \setminus bp(L), A', L$ and $D'$, the colors $\mu(z)$, $\mu(q)$ and $\mu(a')$ used by $\phi$ are pairwise distinct.
	
		\smallskip
		\textit{Case 1:} There exists $v\in N_2(q)$ such that $\phi(v)=\phi(q)=\mu(z)$.  		
		
		In this case $q$ is the only vertex with color $\mu(z)$ under $\phi$ that has a different color under $\psi_z$.
		Since $v\neq q$, $\psi_z(v)=\phi(v)=\mu(z)$.
		Since $\mu(z)$ is not a color that was recolored to some vertex during the construction of $\psi_z$, we have $\mu(v)=\psi_z(v)=\mu(z)$. By Lemma~\ref{lemma:n2adjacentuts}(b), $v\in N(\{u, s\})$.
		If $v\in N(s)$, then $v=z$, which is a contradiction as $z\notin N_2[q]$.
		Thus, $v\in N(u)$, which implies $v=bp(z)$ as $bp(z)$ is the only vertex in $N(u)$ with color $\mu(z)$ under $\mu$.
		However, $\phi(bp(z))=\psi_z(bp(z))=\mu(bn(bp (z)))\neq\mu(z) = \phi(v)$, which is a contradiction.
		
		\smallskip
		\textit{Case 2:} There exists $v\in N_2(a')$ such that $\phi(v)=\phi(a')=\mu(q)$.  		
		
		In this case $a'$ is the only vertex with color $\mu(q)$ under $\phi$ that has a different color under $\psi_z$.
		Since $v\neq a'$, $\psi_z(v)=\phi(v)=\mu(q)$.
		Since $\mu(q)$ is not a color that was recolored to some vertex during the construction of $\psi_z$, we have $\mu(v)=\psi_z(v)=\mu(q)$.
		By Lemma~\ref{lemma:n2adjacentuts}(b), $v\in N(\{u, s\})$. As $\mu(q)\notin \mu(N(s))$ by our assumption, we have
		$v\notin N(s)$.
		So $v\in N(u)$ which implies $v=q$. On the other hand, by the construction of $\phi$, we have $\phi(q)=\mu(z)\neq \mu(q)$, which means $\phi(q)\neq \phi(v)$, which is a contradiction to $v=q$.
	
	\smallskip
	\textit{Case 3:} There exists $v\in N_2(p)$ such that $\phi(v)=\phi(p)=\mu(a')$.  
			
		In this case $p$ is the only vertex with color $\mu(a')$ under $\phi$ that has a different color under $\psi_z$.
		Since $v\neq p$, $\psi_z(v)=\phi(v)=\mu(a')$.
		Since $\mu(a')$ is not a color that was recolored to some vertex during the construction of $\psi_z$, we have $\mu(v)=\psi_z(v)=\mu(a')$. Note that $a'$ is the only vertex in $N_2(p)$ with color $\mu(a')$ under $\mu$,
		which implies that $a'=v$.
		However, $\phi(a')=\mu(q)\neq \mu(a')$, which is a contradiction.
\end{proof}

\subsection{Finale}

Denote by $a_1', a_2', \ldots, a_{k}'$ the vertices in $A'$ and $z_1, z_2, \ldots, z_{\ell}$ the vertices in $D'\cup L$, where $k = |A'|$ and $\ell = |D'\cup L|$.

\smallskip
\textit{Case A:} $k \le \ell$.

In this case, by Lemmas~\ref{lemma:pathd-adj} and~\ref{lemma:qpath}, for each $1 \le i \le k$, we can take a $(\mu,\mu(a_i'),\mu(z_i))$-bichromatic path $P_i$ from $a_i'$ to $\coup(a_i')$. Define $\brset$ to be the family of the following branch sets: each vertex in $N[w]$ is a singleton branch set, each vertex in $F$ is a singleton branch set, and each $V(P_i)$ for $1\le i\le k$ is a branch set.

\smallskip
\textit{Case B:} $\ell < k$.

In this case, by Corollaries \ref{corollary:dadja'}(b) and~\ref{cor:pathq}, for each $1 \le i \le \ell$, we can take a $(\mu,\mu(a'_i),\mu(z_i))$-bichromatic path $P_i$ from $z_i$ to $\coup(z_i)$. Define $\brset$ to be the family of the following branch sets: each vertex in $N[u]\setminus bp(L)$ is a singleton branch set, and each $V(P_i)$ for $1\le i\le \ell$ is a branch set. 

In either case above, the paths $P_1, P_2, \ldots, P_n$ (where $n = \min\{k, \ell\}$) are pairwise vertex-disjoint because the colors of the vertices in $P_i$ and $P_j$ are distinct for $i \neq j$. Therefore, the branch sets in $\brset$ are pairwise disjoint in either case.

\begin{lemma}
\label{lemma:brset-connected}
Each pair of branch sets in $\brset$ are joined by at least one edge in $G^2$.
\end{lemma}
\begin{proof}
Consider Case A first. It is readily seen that $N(w) \cup F$ is a clique of $G^2$. Hence the singleton branch sets in $\brset$ are pairwise adjacent. For $1\le i\le k$, each vertex in $N(w)\cup F$ is adjacent to $a'_i$ or $\coup(a'_i)$ in $G^2$. 
	Hence each singleton branch set is adjacent to each path branch set.
	For $1\le i, j\le k$ with $i \ne j$, we have $a'_j\in N_2[a'_i]$ and thus the branch sets $V(P_i)$ and $V(P_j)$ are joined by at least one edge. 

Now consider Case B. Since $N(u)\setminus bp(L)$ is a clique of $G^2$, the singleton branch sets in $\brset$ are pairwise adjacent.
	All vertices in $N(u)\setminus (bp(L)\cup A' \cup (Q \setminus bp(L)))$ are adjacent to $\coup(z_i)$ in $G^2$ for $1\le i\le \ell$.
	By Corollaries~\ref{qadja'} and~\ref{corollary:dadja'}(a), all vertices in $A'$ are adjacent to $z_i$ in $G^2$ for $1\le i\le \ell$. 
	By Lemma~\ref{lemma:q3conn-q1d'}, all vertices in $Q \setminus bp(L)$ are adjacent to $z_i$ in $G^2$ for $1\le i\le \ell$.
	Hence each singleton branch set is joined to each path branch set by at least one edge.
	Since $z_i\in N(s)$ for $1\le i\le \ell$, the path branch sets are pairwise joined by at least one edge.
\end{proof}

\begin{lemma}
\label{lemma:brsetcolors}
	$|\brset|\ge \chi(G^2)$.
\end{lemma}
\begin{proof}
By Lemma~\ref{lemma:allcolors}, all colors used by $\mu$ are present in $\mu(N_2[p])$. In Case A, all colors in $\mu(N_2[p]) \setminus \mu(A)$ are present in $N(w)\cup F$, the set of singleton branch sets in $\brset$. Hence $|\brset| \ge (|N_2[p]|-|\mu(A)|) + k = (\chi(G^2) - k) +  k = \chi(G^2)$. In Case B, all colors in $\mu(N_2[p]) \setminus \mu(D' \cup L)$ are present in $N(u)\setminus bp(L)$, the set of singleton branch sets in $\brset$. Hence $|\brset| \ge (|N_2[p]|-|\mu(D' \cup L)|)+ \ell=(\chi(G^2)-\ell)+\ell = \chi(G^2)$.
\end{proof}

Theorem~\ref{thm:2tree} follows from Lemmas \ref{lemma:brset-connected} and \ref{lemma:brsetcolors} immediately.

\section{Proof of Corollary \ref{coro:2tree}}
\label{sec:coro2trees}

We now prove Corollary \ref{coro:2tree} using Theorem \ref{thm:2tree}. It can be easily verified that if $G$ is a generalized 2-tree with small order, say at most $4$, then $G^2$ has a clique minor of order $\chi(G^2)$ for which each branch set induces a path. Suppose by way of induction that for some integer $n \ge 5$, for any generalized $2$-tree $H$ of order less than $n$, $H^2$ has a clique minor of order $\chi(H^2)$ for which each branch set induces a path. Let $G$ be a generalized $2$-tree with order $n$. If $G$ is a $2$-tree, then by Theorem \ref{thm:2tree}, the result in Corollary \ref{coro:2tree} is true for $G^2$. Assume that $G$ is not a $2$-tree. Then at some step in the construction of $G$, a newly added vertex $v$ is made adjacent to a single vertex $u$ in the existing graph. (Note that $v$ may be adjacent to other vertices added after this particular step.) This means that $u$ is a cut vertex of $G$. Thus $G$ is the union of two edge-disjoint subgraphs $G_1, G_2$ with $V(G_1)\cap V(G_2) = \{u\}$. Since both $G_1$ and $G_2$ are generalized $2$-trees, by the induction hypothesis, for $i=1,2$, $G_i^2$ has a clique minor of order $\chi(G_i^2)$ for which each branch set induces a path. It is evident that $G^2$ is the union of $G_1^2$, $G_2^2$ and the clique induced by the neighborhood $N_G(u)$ of $u$ in $G$. 
 
Denote $N_i = N_{G_i}(u)$ for $i=1,2$. Then in any proper coloring of $G_i^2$, the vertices in $N_i$ need pairwise distinct colors. Without loss of generality we may assume $\chi(G_1^2) \le \chi(G_2^2)$. If $|N_G(u)| = |N_1|+|N_2| \le \chi(G_2^2) - 1$, then we can color the vertices in $N_1$ using the colors that are not present at the vertices in $N_2$ in an optimal coloring of $G_2^2$. Extend this coloring of $N_1$ to an optimal coloring of $G_1^2$. One can see that we can further extend this optimal coloring of $G_1^2$ to obtain an optimal coloring of $G^2$ using $\chi(G_2^2)$ colors. Thus, if $|N_G(u)| \le \chi(G_2^2) - 1$, then $\chi(G^2) = \chi(G_2^2)$. Moreover, the above-mentioned clique minor of $G_2^2$ is a clique minor of $G^2$ with order $\chi(G^2)$ for which each branch set induces a path. On the other hand, if $|N_G(u)| \ge \chi(G_2^2)$, then one can show that $\chi(G^2) = |N_G(u)|$ and $N_G(u)$ induces a clique minor of $G^2$, with each branch set a singleton. In either case we have proved that $G^2$ has a clique minor of order $\chi(G^2) = \max\{\chi(G_1^2), \chi(G_2^2), |N_G(u)|\}$ for which each branch set induces a path. This completes the proof of Corollary \ref{coro:2tree}. 
 
\section{Concluding remarks} 
\label{sec:rem}

\begin{figure}
	\centering
	\includegraphics[scale=0.3]{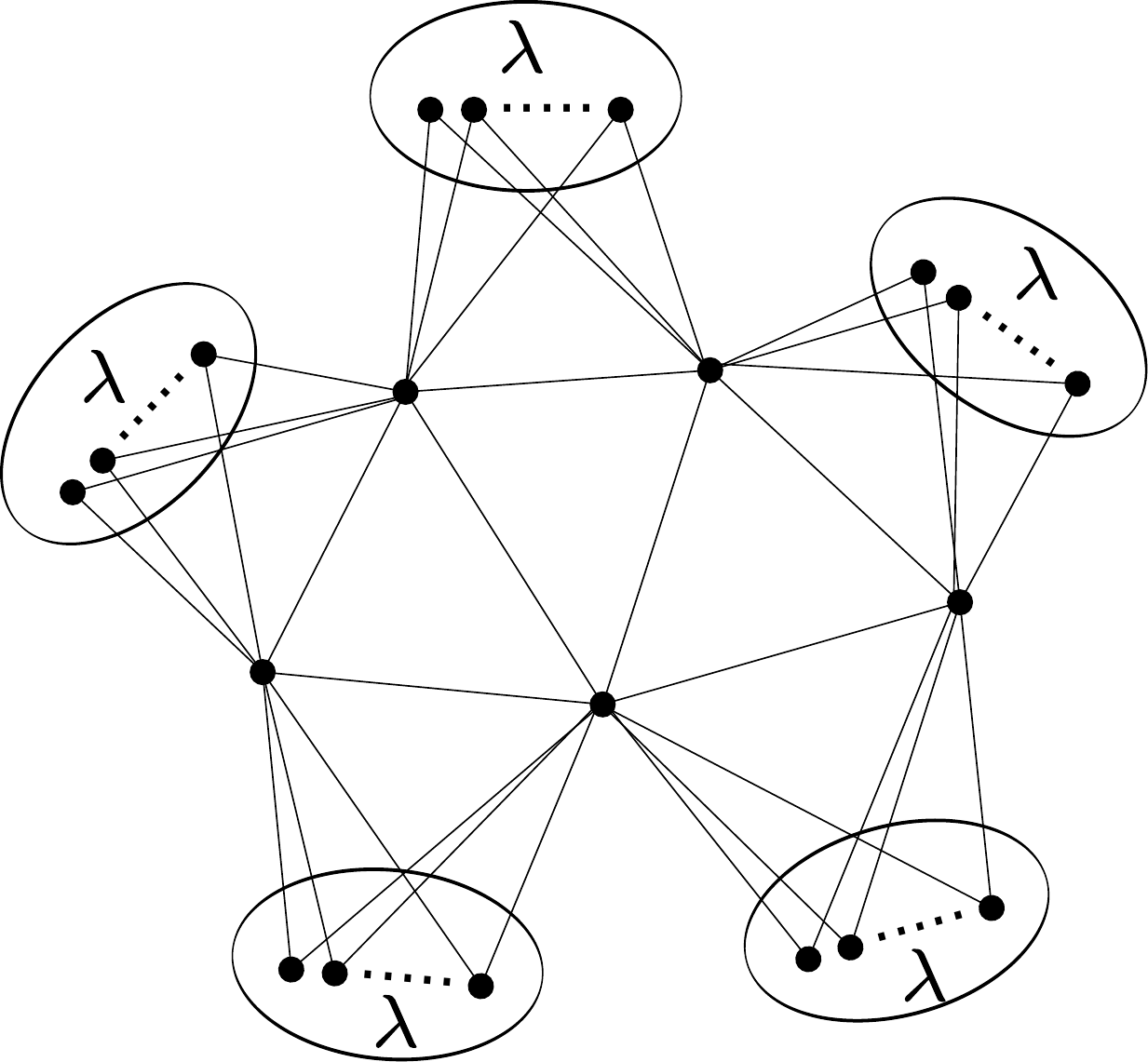}
	\caption{A 2-tree $G$ with $\omega(G^2)=2\lambda+5$ and $\chi(G^2)=3\lambda+3$.} 
	\label{fig:counter-eg}
\end{figure}

We have proved that for any $2$-tree $G$, $G^2$ has a clique minor of order $\chi(G^2)$. Since large cliques played an important role in our proof of this result, it is natural to ask whether $G^2$ has a clique of order close to $\chi(G^2)$, say, $\omega(G^2) \ge c \chi(G^2)$ for a constant $c$ close to $1$ or even $\omega(G^2) = \chi(G^2)$. Since the class of $2$-trees contains all maximal outerplanar graphs, this question seems to be relevant to Wegner's conjecture \cite{W}, which asserts that for any planar graph $G$ with maximum degree $\Delta$, $\chi(G^2)$ is bounded from above by $7$ if $\Delta=3$, by $\Delta + 5$ if $4\le\Delta\le 7$, and by $(3\Delta/2) + 1$ if $\Delta\ge 8$. For $\Delta=3$, this conjecture has been proved by Thomassen in \cite{THOMASSEN2017}. In the case of outerplanar graphs with $\Delta=3$, a stronger result holds as shown by Li and Zhou in \cite{LZ}. In \cite{Lih2003303}, Lih, Wang and Zhu proved that for any $K_4$-minor free graph $G$ with $\Delta\geq 4$, $\chi(G^2)\le (3\Delta/2) + 1$. Since $2$-trees are $K_4$-minor free, this bound holds for them. Combining this with $\omega(G^2)\ge \Delta(G)$, we then have $\omega(G^2) \ge 2(\chi(G^2)-1)/3$ for any $2$-tree $G$. It turns out that the factor $2/3$ here is the best one can hope for: In Figure~\ref{fig:counter-eg}, we give a $2$-tree whose square has clique number $2\lambda+5$ and chromatic number $3\lambda+3$. 
 
In view of Theorem \ref{thm:2tree}, the obvious next step would be to prove Hadwiger's conjecture for squares of $k$-trees for a fixed $k\ge 3$. Since squares of $2$-trees are 2-simplicial graphs, another related problem would be to prove Hadwiger's conjecture for the class of 2-simplicial graphs or some interesting subclasses of it. It is also interesting
to work on Hadwiger's conjecture for squares of some other special classes of graphs such as planar graphs.

\vspace{-0.1in}

\begin{small}
\bibliographystyle{plain}
\bibliography{hadwiger}
\end{small}

\end{document}